\documentclass[12pt]{amsart}

\usepackage[utf8]{inputenc}
\usepackage[english]{babel}

\usepackage{amsfonts, amsmath, amssymb, amsthm}
\usepackage{mathtools}
\usepackage{mathdots}
\usepackage{faktor}
\usepackage{enumerate}
\usepackage[inline]{enumitem}
\usepackage{enumitem}
\usepackage{varwidth}
\usepackage{tasks}

\usepackage{csquotes}
\usepackage[
backend=biber
]{biblatex}
\usepackage{hyperref}

\theoremstyle{plain}
\newtheorem*{definition}{Definition}

\theoremstyle{definition}

\theoremstyle{definition}

\theoremstyle{plain}
\newtheorem{theorem}{Theorem}

\theoremstyle{plain}
\newtheorem*{theorem*}{Theorem}

\theoremstyle{plain}
\newtheorem{lemma}{Lemma}

\theoremstyle{plain}

\theoremstyle{plain}
\newtheorem*{corollary*}{Corollary}

\theoremstyle{plain}

\title{Associative Rota--Baxter operators on the Sweedler algebra~$H_4$}

\author{Maxim V.~Podkorytov}

\date{\today}

\begin{document}

\begin{abstract}
In this paper, we classify all Rota--Baxter operators on the Sweedler algebra \( H_4 \) up to conjugation and dualization. Modulo algebra (anti)automorphisms of \( H_4 \), we first describe its subalgebras and then analyse the kernel of a Rota--Baxter operator. The classification is carried out according to the dimension of this kernel, yielding a complete description of such operators. A complete list of operators is given in the theorem of the final section.
\end{abstract}
\maketitle
\tableofcontents

\section*{Introduction}

Rota–Baxter operators first appeared in a 1951 paper by F. Tricomi \cite{tricomi1951}. The concept was subsequently developed by G. E. Baxter in his 1960 work \cite{baxter1960}, where it arose from the study of fluctuation theory in probability and provided an algebraic foundation for Spitzer’s identity. Ten years later, Gian-Carlo Rota recognized its profound combinatorial significance, placing this operator at the heart of an algebraic theory of summation \cite{rota1969}. The identity  
\[
P(x)P(y) = P\bigl(P(x)y + xP(y) + \lambda xy\bigr),
\]
now known as the Rota--Baxter identity of weight \(\lambda\), has since become a fundamental structure in several areas of mathematics. The first monograph devoted to Rota–Baxter algebras appeared in 2011 \cite{guo:intro}.

In recent years, the theory of Rota--Baxter operators has been actively developing. In 2020, the notion of a Rota--Baxter operator on a group was introduced \cite{guo2021}, and in the same year, a definition of a Rota--Baxter operator on a cocommutative Hopf algebra was given \cite{goncharov2021}. We also mention recent classification results devoted to the description of such operators on specific algebras: on $M_2(F)$ see \cite{benito2018}, on $M_3(F)$ see \cite{goncharov2020}. Of course, this is not an exhaustive list of works devoted to classification; we hope that readers interested in this topic will find further references in the cited papers.

The Sweedler algebra $H_4$ is the four-dimensional algebra generated by two elements $x$ and $g$ subject to the relations  
\[
x^2 = 0, \quad g^2 = 1, \quad \text{and} \quad gx = -xg.
\]  
A standard basis is $\{1, g, x, gx\}$.

We begin with the main definition. Let $F$ be a field with $\mathrm{char}(F) \neq 2$, and let $A$ be a $F$-algebra.

\begin{definition}
A linear operator $R: A \to A$ is called a {\normalfont\bf Rota--Baxter operator of weight} $\lambda \in F$ if  
\begin{equation*} \label{idn:RB} \tag{RB}
R(a)R(b) = R(R(a)b + aR(b) + \lambda ab)
\end{equation*}
for all $a, b \in A$.
\end{definition}

The zero operator and the operator $-\lambda \operatorname{id}$ are called trivial. We recall the following basic properties of Rota--Baxter operators \cite{guo:intro}. Let $R$ be a Rota--Baxter operators of weight $\lambda$. Then:
\begin{enumerate*}[itemjoin={\,}]
\item the operator $-\lambda \operatorname{id} - R$ is also a Rota--Baxter operator of weight $\lambda$;

\item for any (anti)automorphism $\varphi \colon A \to A$, the operator $\varphi^{-1}R\varphi$ is again a Rota--Baxter operator of weight $\lambda$.
\end{enumerate*}

In the paper \cite{ma:r-bo}, a description of Rota--Baxter operators on the algebra $H_4$ was given:
\begin{enumerate}[label=(\alph*),itemjoin={\\*[\smallskipamount]}]
\item 
$
R(1) = 0, \ R(g) = 0, \ R(x) = -\lambda x, \ R(gx) = -\lambda gx.
$ \medskip

\item 
$
R(1) = -\lambda 1, \ R(g) = -\lambda g, \ R(x) = 0, \ R(gx) = 0.
$ \medskip

\item 
$
R(1) = -\lambda 1, \ R(g) = -\lambda g, \ R(x) = -\lambda x, \ R(gx) = -\lambda gx.
$ \medskip

\item $ R(1) = 0, \ R(g) = -p_1\, 1 + p_1\, g - \frac{(\lambda + p_1)(\lambda + p_1 + p_2)}{p_3}\, x + \frac{(\lambda + p_1)(\lambda + p_2)}{p_3}\, gx, \\ 
R(x) = -p_3\, 1 + p_3\, g - (2\lambda + p_1 + p_2)\, x + (\lambda + p_2)\, gx, \\ 
R(gx) = -p_3\, 1 + p_3\, g - (\lambda + p_1 + p_2)\, x + p_2\, gx. $ \medskip

\item $
R(1) = -\lambda 1, \
R(g) = (\lambda + p_1)\, 1 + p_1\, g - \frac{(\lambda + p_1)(\lambda + p_1 + p_2)}{p_3}\, x + \frac{(\lambda + p_1)(\lambda + p_2)}{p_3}\, gx, \\
R(x) = p_3\, 1 + p_3\, g - (2\lambda + p_1 + p_2)\, x + (\lambda + p_2)\, gx, \\
R(gx) = p_3\, 1 + p_3\, g - (\lambda + p_1 + p_2)\, x + p_2\, gx.
$ \medskip

\item $ 
R(1) = -\lambda 1, \
R(g) = \lambda 1 + p_1\, x + \frac{p_1 p_2}{\lambda + p_2}\, gx, \\
R(x) = -(\lambda + p_1)\, x - p_2\, gx, \\
R(gx) = (\lambda + p_2)\, x + p_2\, gx. $ \medskip 

\item $
R(1) = -\lambda 1, \
R(g) = \lambda 1 + \frac{\lambda(\lambda + p_1)}{p_2}\, x + \frac{\lambda(\lambda + p_1)}{p_2}\, gx, \\
R(x) = -p_2 1 -p_2\, g - (2\lambda + p_1)\, x - (\lambda + p_1)\, gx, \\
R(gx) = p_2 1 + p_2\, g + (\lambda + p_1)\, x + p_1\, gx. $ \medskip

\item $
R(1) = \frac{\lambda}{2}\, 1 - \frac{\lambda}{2}\, g + p_1\, x + p_2\, gx, \
R(g) = \frac{\lambda}{2}\, 1 - \frac{\lambda}{2}\, g - p_2\, x + p_1\, gx, \\
R(x) = -\frac{\lambda}{2}\, x - \frac{\lambda}{2}\, gx, \
R(gx) = -\frac{\lambda}{2}\, x - \frac{\lambda}{2}\, gx.$
\end{enumerate}

The present article is organized as follows. In the first section, we describe all subalgebras of dimensions 3 and 2. In the subsequent sections, we classify Rota--Baxter operators on the Sweedler algebra according to the dimension of their kernel. In the final section, we compute the dual operators and present a complete list of Rota--Baxter operators up to conjugation and dualization. Note that the operators (h) is a Rota--Baxter operator if and only if $p_1 = 0$. 

\section{Subalgebras of $H_4$}

Let $A$ be a $F$-algebra. If \( S \subset A \), we denote by \( \langle S \rangle \) the subalgebra generated by the set \( S \). Also if elements $e_1, \dots , e_n \in A$ are linearly independent and $S = \{e_1, \dots , e_n\}$ we just write
$$
\langle S \rangle
=
\langle e_1, \dots , e_n \rangle
\quad
\text{or}
\quad
\langle S \rangle
=
Fe_1 + \dots +Fe_n.
$$
We now proceed to the description of subalgebras.

\begin{lemma}
Let $L \subset H_4$ be a $3$-dimensional subalgebra. Then $L$ is one of the following spaces:
\[
\langle 1 + \sigma g, x, gx \rangle, \quad
\langle 1, g + y_3 x + y_4 gx, x + \sigma gx \rangle, \quad
\langle 1, x, gx \rangle,
\]
where $\sigma \in F$ and $\sigma^2 = 1$. 
\end{lemma}

\begin{proof}
Let $I = F1 + Fx + Fgx$. There are two cases: either $\dim (L \cap I) = 2$ or $\dim (L \cap I) = 3$. Suppose that $\dim (L \cap I) = 2$. Therefore, there are exists two nonzero elements $e_1, e_2 \in L \cap I$ that span $L \cap I$. Choose $e_3 = y_1 1 + y_2 g + y_3 x + y_4 gx \in H_4$ such that $L = \langle e_1, e_2, e_3 \rangle$. If one of the elements $e_1, e_2$ is equal to 1, say $e_1 = 1$, then $y_2 \neq 0$; otherwise contradicting $\dim (L \cap I) = 2$. Let $e_2 = x_3 x + x_4 gx$. Since $e_3 e_2 = x_4 x + x_3 gx$ belongs to $L$, then $x_4 = \sigma x_3, \ x_3 = \sigma x_4$. So, $e_2 = x + \sigma gx, \ \sigma^2 = 1$. 

Now suppose that $1 \notin L \cap I$. Again $y_2 \neq 0$; and $y_1 \neq 0$, otherwise $e_3^2 = 1$, but $1 \notin L$. So, $e_3 = y_1 1 + g + y_3 x + y_4 gx$ and $e^2_3 = (y_1^2 + 1) 1 + 2y_1 g + 2 y_1 y_3 x + 2 y_1 y_4 gx$. Therefore, $y_1^2 + 1 = \sigma y_1, \ 2y_1 = \sigma$, consequently, $y_1^2 = 1$. Let $e_2 = x_1 1 + x_3 x + x_4 gx$. Since the element $e_3e_2 - e_2e_3 = 2x_4 x + 2x_3 gx$ belongs to $L$ and $1 \notin L \cap I$, then $e_1, e_2 \in \langle x, gx \rangle$. We can put $e_1 = x$ and $e_2 = gx$. Thus, we obtain subalgebra $\langle 1 + \sigma g, x, gx \rangle, \ \sigma^2 = 1$. 

If $\dim (L \cap I) = 3$, then $L = I = \langle 1,x,gx \rangle$. This conclude the proof. 
\end{proof}

\begin{lemma}
Let $L \subset H_4$ be a $2$-dimensional subalgebra. Then $L$ is one of the following spaces:
\[
\langle 1 + \sigma g + y_3 x + y_4 gx, x + \mu gx \rangle, \quad 
\langle x, gx \rangle, \quad 
\langle 1, x_2 g + x_3 x + x_4 gx \rangle,
\]
where $\sigma, \mu \in F$ and $\sigma^2 = 1, \mu^2 = 1$. 
\end{lemma}

\begin{proof}
Let $I = F1 + Fx + Fgx$. There are two possibilities: either $\dim(L \cap I) = 1$ or $\dim(L \cap I) = 2$. Assume first that $\dim(L \cap I) = 1$. Then there exists a nonzero element $e_1 \in L \cap I$. If $e_1 = 1$, then clearly $L = \langle 1, x_2 g + x_3 x + x_4 gx \rangle$. Otherwise, write $e_1 = x_1 1 + x_3 x + x_4 gx$ with $e_1 \ne 1$. Choose $e_2 = y_1 1 + y_2 g + y_3 x + y_4 gx \in H_4$ such that $L = \langle e_1, e_2 \rangle$. We may assume $y_2 \ne 0$ (and normalize $y_2 = 1$); otherwise $L \subseteq I$, contradicting $\dim(L \cap I) = 1$.

Now, the condition $e_1^2 = \alpha e_1$ yields
\[
\begin{cases*}
x_1(x_1 - \alpha) = 0, \\
x_3(2x_1 - \alpha) = 0, \\
x_4(2x_1 - \alpha) = 0,
\end{cases*}
\quad \Longleftrightarrow \quad x_1 = \alpha = 0,
\]
the last equivalence following from $e_1 \ne 1$. Similarly, expanding $e_2^2$ in the basis of $L$ gives $e_2^2 = \beta e_2$, which forces $e_2 = 1 + \sigma g + y_3 x + y_4 gx$ with $\sigma^2 = 1$. Finally, requiring $e_1 e_2 \in L$ implies $e_1 = x + \mu gx$ with $\mu^2 = 1$.

Now suppose $\dim(L \cap I) = 2$. If $1 \in L$, we obtain $L = \langle 1, x_2 x + x_3 gx \rangle$, which is already covered by the previous case. If $1 \notin L$, then $L = \langle x, gx \rangle$.
\end{proof}

Let $\operatorname{Aut}(H_4)$ be automorphism group of $H_4$ (as associative algebra). Show that any $\varphi \in \operatorname{Aut}(H_4)$ acts on the generators $x$ and $g$ as follows:
\[
\begin{array}{rl}
\varphi(g) = \varepsilon g + a x + b gx, \ & \varphi(x) = p x + q gx, \\[.5em]
a,b,p,q,\varepsilon \in F, \ & \varepsilon = \pm 1, \ p^2 - q^2 \neq 0.
\end{array}
\]
Since $\varphi(1) = 1$, then $\varphi(g)\varphi(g) = 1$. So, $\varphi(g) = \varepsilon g + a x + b gx$, $a,b,\varepsilon \in F$ and $\varepsilon = \pm 1$. It's quite obvious that $\varphi(x)\varphi(x) = 0$. Thus $\varphi(x) = p x + q gx, p, q \in F$. Finally $\varphi(gx) = \varphi(g)\varphi(x) = \varepsilon q x + \varepsilon p gx$ and $\det (\varphi) = p^2 - q^2 \neq 0$.  

We now show that, up to isomorphism, there are exactly three distinct $3$-dimensional subalgebras of $H_4$. Indeed, the automorphism defined by $\varphi(g) = -g$, $\varphi(x) = x$ identifies $\langle 1+g, x, gx \rangle$ with $\langle 1-g, x, gx \rangle$. The subalgebra $\langle 1,g,x-gx \rangle$ isomorphic to $\langle 1,g+y_3x+y_4gx,x+\sigma gx \rangle$ under the action of automorphism:
$$
f(g) = g + y_3 x + y_4 gx, \ f(x) = \dfrac{1+\sigma}{2} x + \dfrac{\sigma-1}{2} gx. 
$$

Turning to $2$-dimensional subalgebras, the algebras $\langle 1 + \sigma g + y_3 x + y_4 gx,\, x - \sigma gx \rangle$ are isomorphic to $\langle 1-g, x+gx \rangle$ via the automorphisms $f$ given by
\[
\begin{array}{rl}
f(g) = -\sigma g + a x + b x, & f(x) = (1-q)x + q gx, \\[.5em]
a = (y_4 + \sigma y_3)q - \sigma y_3, & b = y_4 - (y_4 + \sigma y_3)q, \quad q \neq 2^{-1}.
\end{array}
\]
Similarly, the subalgebras $\langle 1 + \sigma g + y_3 x + y_4 gx,\, x + \sigma gx \rangle$ are isomorphic to $\langle 1-g, x-gx \rangle$ under the automorphisms $\widetilde{f}$ defined by
\[
\begin{array}{rl}
\widetilde{f}(g) = -\sigma g + a x + b gx, & \widetilde{f}(x) = (1+q)x + q gx, \\[.5em]
a = (y_4 - \sigma y_3)q - \sigma y_3, & b = y_4 + (y_4 - \sigma y_3)q, \quad q \neq -2^{-1}.
\end{array}
\]
Note that subalgebra $\langle 1-g, x-gx \rangle$ is isomorphic to $\langle 1-g, x+gx \rangle$ under the antiautomorphism $\psi(g) = g, \ \psi(x) = x$. 

It remains to consider subalgebras of the form $\langle 1, x_2 g + x_3 x + x_4 gx \rangle$. Since $\langle 1, g \rangle \cong \langle 1, g + a x + b gx \rangle$ for any $a, b \in k$, it follows that if $x_2 \ne 0$, then $\langle 1, x_2 g + x_3 x + x_4 gx \rangle \cong \langle 1, g \rangle$. If $x_2 = 0$, then the subalgebra is isomorphic either to $\langle 1, x + C gx \rangle$ or to $\langle 1, C x + gx \rangle$. In fact, these two are isomorphic to each other, and we have
\begin{align*}
\langle 1, x + C gx \rangle \cong \langle 1, x - gx \rangle &\iff C^2 = 1, \\
\langle 1, x + C gx \rangle \cong \langle 1, x \rangle &\iff C^2 \ne 1.
\end{align*}
Thus, we have obtained the following lemma.

\begin{lemma}
Up to isomorphism, \(H_4\) contains the following $2$-dimensional subalgebras: 
\[
\langle 1 - g, x - gx \rangle, \quad 
\langle 1, g \rangle, \quad 
\langle 1, x - gx \rangle, \quad 
\langle 1, x \rangle, \quad 
\langle x, gx \rangle;
\]
and the following $3$-dimensional subalgebras: 
\[
\langle 1 - g, x, gx \rangle, \quad
\langle 1, g, x - gx \rangle, \quad
\langle 1, x, gx \rangle.
\]
\end{lemma}

We are now ready to proceed with the classification of Rota--Baxter operators.

\section{Dimension of kernel is 3}

Here and throughout, we consider operators of nonzero weight ($\lambda \neq 0$).

\begin{theorem}
Let $\ker R = \langle 1-g, x, gx \rangle$. Then we have the following operators:\\*[\smallskipamount]
\begin{enumerate*}[label=\normalfont(\arabic*),itemjoin={\quad}]
\item $R(g) = -\lambda 1$,

\item $R(g) = -\frac{\lambda}{2} 1 - \frac{\lambda}{2} g + \gamma_g x + \delta_g gx$, 
	
\item $R(g) = \frac{\lambda}{2} 1 - \frac{\lambda}{2} g + \gamma_g x + \delta_g gx$.
\end{enumerate*}
\end{theorem}

\begin{proof}
Let $R(g) = \alpha_g 1 + \beta_g g + \gamma_g x + \delta_g gx$. Consider the Rota--Baxter identity for the pair $(g,g)$: 
\begin{align*}
\begin{array}{rllrll}
\alpha_g (2\beta_g + \lambda) &=& \beta_g^2 - \alpha_g^2, &
\beta_g (2\beta_g + \lambda) &=& 0, \\
\gamma_g (2\beta_g + \lambda) &=& 0, &
\delta_g (2\beta_g + \lambda) &=& 0. 
\end{array}
\end{align*}
Obviously, if $2\beta_g + \lambda = 0$, then we get the operators (2) and (3). Otherwise, if $2\beta_g + \lambda \ne 0$, then $\beta_g = \gamma_g = \delta_g = 0$ and $\alpha_g + \lambda = 0$. Thus we get the operator (1). 

Consider the automorphism 
$$
\varphi(g) = g - \dfrac{2\gamma_g}{\lambda} x - \dfrac{2\delta_g}{\lambda} gx, \ \varphi(x) = x.
$$ 
Let's show that the operators (2) and (3) with arbitrary $\gamma_g$ and $\delta_g$ are conjugate to the operators (2) and (3) with $\gamma_g = \delta_g = 0$. Compute the automorphism $\varphi^{-1}$:
$$
\varphi^{-1}(g) = g + \dfrac{2\gamma_g}{\lambda} x + \dfrac{2\delta_g}{\lambda} gx, \ \varphi^{-1}(x) = x.
$$
It's clear that if $R$ is the operator (2) or (3) then $\varphi^{-1}R\varphi(g) = \pm \frac{\lambda}{2} 1 - \frac{\lambda}{2} g.$
\end{proof}

\begin{theorem}
Let $\ker R = \langle 1, x, gx \rangle$. Then we have the following operators:\\*[\smallskipamount]
\begin{enumerate*}[itemjoin={\quad}]
\item[\normalfont(1a)] $R(g) = -\lambda 1 - \lambda g + \gamma_g x + \delta_g gx$, 

\item[\normalfont(1b)] $R(g) = \lambda 1 - \lambda g + \gamma_g x + \delta_g gx$.
\end{enumerate*}
\end{theorem}

\begin{proof}
Obviously, $0 = R(1)R(g) = R(R(g) + \lambda g)$. So, $R(g) + \lambda g \in \ker R$. Applying \eqref{idn:RB} to the pair $(g,g)$, we get $\alpha_g^2 = \beta_g^2$. Note that (1a) and (1b) lie in the same orbit under the action of the automorphism:
$$
f(g) = -g + \dfrac{2\gamma_g}{\lambda} x, \quad
f(x) = x. 
$$
Indeed, firstly, note that $f = f^{-1}$; secondly, if $R$ is the operator (1a) then
\begin{align*}
f^{-1}Rf(g) &= fRf(g) = fR(-g + \frac{2\gamma_g}{\lambda} x) \\
&= f(\lambda 1 + \lambda g - \gamma_g x - \delta_g gx) \\
&= \lambda 1 - \lambda g + \gamma_g x + \delta_g gx.
\end{align*}
So, we get the operator (1b).

Also note that under the automorphism
$$
\varphi(g) = g - \dfrac{\gamma_g}{\lambda} x - \dfrac{\delta_g}{\lambda} gx, \ \varphi(x) = x,
$$
the operator (1a) with arbitrary $\gamma_g$ and $\delta_g$ is conjugate to the operator (1a) with $\gamma_g = \delta_g = 0$. This concludes the proof. 
\end{proof}

\begin{theorem}
Let $\ker R = \langle 1, g, x-gx \rangle$. Then we have the following operators:\\*[\smallskipamount]
\begin{enumerate*}[label=\normalfont(\roman*),itemjoin={\quad}]
\item[\normalfont (1a)] $R(gx) = \alpha_{gx} 1 - \alpha_{gx} g + \gamma_{gx} x - (\lambda + \gamma_{gx}) gx$,

\item[\normalfont (1b)] $R(gx) = \alpha_{gx} 1 + \alpha_{gx} g + \gamma_{gx} x - (\lambda + \gamma_{gx}) gx$.
\end{enumerate*}
\end{theorem}

\begin{proof}
Applying \eqref{idn:RB} to the pairs $(1,gx)$ and $(gx,gx)$, we get $R(gx) + \lambda gx \in \ker R$ and $\alpha_{gx}^2 = \beta_{gx}^2$, respectively. So, we have obtained the operators (1a) and (1b). Note that under the action of the antiautomorphism $\psi(g) = -g, \ \psi(x) = x$ the operators (1a) and (1b) are conjugated. This concludes the proof. 
\end{proof}

\section{Dimension of kernel is 2}

\begin{theorem}
Let $\ker R = \langle 1, g \rangle$. Then we have the following operators:\\*[\smallskipamount]
\begin{enumerate*}[label=\normalfont(\roman*),itemjoin={}]
\item[\normalfont (1a)]
$
\begin{array}{lcl}
R(x)&=& \alpha_x 1 + \alpha_x g - \lambda x, \\
R(gx)&=& \alpha_x 1 + \alpha_x g - \lambda gx.
\end{array}
$ \quad \quad

\item[\normalfont (1b)]
$
\begin{array}{lcl}
R(x)&=& \alpha_x 1 - \alpha_x g - \lambda x, \\
R(gx)&=& -\alpha_x 1 + \alpha_x g - \lambda gx.
\end{array}
$\\*[\medskipamount]

\item[\normalfont (1c)]
$
\begin{array}{lcl}
R(x)&=& \alpha_x 1 - \alpha_x g - \lambda x, \\
R(gx)&=& \alpha_x 1 - \alpha_x g - \lambda gx.
\end{array}
$ \quad \quad

\item[\normalfont (1d)]
$
\begin{array}{lcl}
R(x)&=& \alpha_x 1 + \alpha_x g - \lambda x, \\
R(gx)&=& -\alpha_x 1 - \alpha_x g - \lambda gx.
\end{array}
$
\end{enumerate*}
\end{theorem}
\begin{proof}
Let $R(x) = \alpha_x 1 + \beta_x g + \gamma_x x + \delta_x gx$, $R(gx) = \alpha_{gx} 1 + \beta_{gx} g + \gamma_{gx} x + \delta_{gx} gx$. We claim that $R(x) + \lambda x = \alpha_x 1 + \beta_x g$ and $R(gx) + \lambda gx = \alpha_{gx} 1 + \beta_{gx} g$. Indeed, this follows from \eqref{idn:RB}, applied to the pairs $(1, x)$ and $(1, gx)$: 
\begin{align*}
R(R(x) + \lambda x) = 0, \quad R(R(gx) + \lambda gx) = 0.
\end{align*}
Applying \eqref{idn:RB} to the pairs $(x,x)$ and $(gx,gx)$, we get $\alpha_x^2 = \beta_x^2$ and $\alpha_{gx}^2 = \beta_{gx}^2$, respectively. Further, from \eqref{idn:RB}, applied to the pairs $(x, gx)$ and $(gx, x)$, we have
$$
\begin{cases}
	(\alpha_x - \beta_{gx})(\alpha_{gx} + \beta_x) = 0, \\
	(\alpha_x + \beta_{gx})(\alpha_{gx} - \beta_x) = 0, \\
	\beta_x^2 - \beta_{gx}^2 = 0.
\end{cases}
$$
So, from this system we get four operators. Note that the operators (1a), (1b) and (1c), (1d) lie in the same orbit under the action of the automorphism $\varphi(g) = -g$, $\varphi(x) = x$. Also, the operators (1a) and (1c) are conjugate under the action of the antiautomorphism $\psi(g) = g, \ \psi(x) = x$. 
\end{proof}

\begin{theorem} \label{thr:1,x}
Let $\ker R = \langle 1, x \rangle$. Then we have the following operators:\\*[\smallskipamount]
\begin{enumerate*}[label=\normalfont(\roman*),itemjoin={}]
\item[\normalfont (1a)]
$
\begin{array}{lcl}
	R(g) &=& -\lambda 1 - \lambda g + \gamma_g x, \\
	R(gx) &=& \lambda x - \lambda gx.
\end{array}
$ \quad \quad

\item[\normalfont (1b)]
$
\begin{array}{lcl}
	R(g) &=& \lambda 1 - \lambda g + \gamma_g x, \\
	R(gx) &=& -\lambda x - \lambda gx.
\end{array}
$\\*[\medskipamount]

\item[\normalfont (1c)] 
$
\begin{array}{lcl}
	R(g) &=& -\lambda 1 - \lambda g + \gamma_g x, \\
	R(gx) &=& -\lambda x - \lambda gx.
\end{array}
$ \quad \quad

\item[\normalfont (1d)]
$
\begin{array}{lcl}
	R(g) &=& \lambda 1 - \lambda g + \gamma_g x, \\
	R(gx) &=& \lambda x - \lambda gx.
\end{array}
$
\end{enumerate*}
\end{theorem}
\begin{proof}
Let $R(g) = \alpha_g 1 + \beta_g g + \gamma_g x + \delta_g gx$ and $R(gx) = \alpha_{gx} 1 + \beta_{gx} g + \gamma_{gx} x + \delta_{gx} gx$. First, obviously, $R(g) = \alpha_g 1 - \lambda g + \gamma_g x$ and $R(gx) = \alpha_{gx} 1 + \gamma_{gx} x - \lambda gx$. Secondly, applying \eqref{idn:RB} to the pairs $(g,g)$ and $(gx,gx)$, we have $\alpha_g^2 = \lambda^2$ and $\alpha_{gx} = 0$, respectively. Finally, the Rota--Baxter identity applied to the pair $(g,gx)$ gives us $\gamma_{gx}^2 = \lambda^2$. Combining all ratios, we obtain operators (1a), (1b), (1c), (1d).  

Finally, note that the operators (1a), (1b) and (1c), (1d) lie in the same orbit under the action of the automorphism \(\psi(g) = -g\), \(\psi(x) = -x\). Also, the operators (1a) and (1c) conjugate under the action of antiautomorphism $f(g) = g, \ f(x) = x$. Moreover, one may assume \(\gamma_g = 0\), since the operator in case (1a) with arbitrary \(\gamma_g\) is conjugate to the operator in case (1a) with \(\gamma_g = 0\) via an automorphism:
\begin{align*}
&\varphi(g) = g - \dfrac{\gamma_g}{\lambda} x, \ \varphi(x) = x. \qedhere
\end{align*}
\end{proof}

\begin{theorem} \label{thr:1,x-gx}
Let \(\ker R = \langle 1, x - gx \rangle\). Then we have the following operators:\\*[\smallskipamount]
\begin{enumerate*}[label=\normalfont(\roman*),itemjoin={}]
\item[\normalfont (1a)] 
$
\begin{array}{lcl}
R(g) &=& -\lambda 1 - \lambda g + \gamma_g x - \gamma_g gx, \\
R(gx) &=& -\frac{\lambda}{2} x - \frac{\lambda}{2} gx.
\end{array}
$ \quad \quad

\item[\normalfont (1b)]
$
\begin{array}{lcl}
R(g) &=& \lambda 1 - \lambda g + \gamma_g x - \gamma_g gx, \\
R(gx) &=& -\frac{\lambda}{2} x - \frac{\lambda}{2} gx.
\end{array}
$
\end{enumerate*}
\end{theorem}

\begin{proof}
Since $1 \in \ker R$, let $R(g) = \alpha_g 1 - \lambda g + \gamma_g x - \gamma_g gx$ and $R(gx) = \alpha_{gx} 1 + \gamma_{gx} x - (\gamma_{gx} + \lambda) gx$. Obviously, $\alpha_g^2 = \lambda^2, \ \alpha_{gx} = 0$, it follows from the Rota--Baxter identity applied to the pairs $(g,g), \ (gx,gx)$. Finally, the Rota--Baxter identity applied to the pair $(g,gx)$ gives us $\gamma_{gx} = -\lambda / 2$. So, we have obtained the operators (1a) and (1b). Under the action of the antiautomorphism $\psi(g) = -g, \ \psi(x) = -x$, the operators (1a) and (1b) are conjugated. 

We note that we may assume \(\gamma_g = 0\), since the operator (1a) with arbitrary \(\gamma_g\) is conjugate to (1a) with \(\gamma_g = 0\) under the action of the automorphism:
\begin{align*}
&\varphi(g) = g - \dfrac{2\gamma_g}{\lambda} x, \ \varphi(x) = x. 
\qedhere
\end{align*}
\end{proof}

\begin{theorem} \label{thr:x,gx}
Let $\ker R = \langle x, gx \rangle$. Then we have the following operators:\\*[\smallskipamount]
\begin{enumerate*}[label=\normalfont(\arabic*),itemjoin={}]
\item 
$
\begin{array}{lcl}
	R(1) &=& -\lambda 1, \\
	R(g) &=& -\lambda g + \gamma_{gx} x + \delta_{gx} gx.
\end{array}
$ \quad \quad

\item[\normalfont (2a)]
$
\begin{array}{lcl}
	R(1) &=& -\frac{3\lambda}{2} 1 - \frac{\lambda}{2} g + \gamma_{gx} x + \delta_{gx} gx, \\
	R(g) &=& \frac{\lambda}{2} 1 - \frac{\lambda}{2} g + \gamma_{gx} x + \delta_{gx} gx.
\end{array}
$\\*[\medskipamount]

\item[\normalfont (2b)]
$
\begin{array}{lcl}
	R(1) &=& -\frac{3\lambda}{2} 1 + \frac{\lambda}{2} g - \gamma_{gx} x - \delta_{gx} gx, \\
	R(g) &=& -\frac{\lambda}{2} 1 - \frac{\lambda}{2} g + \gamma_{gx} x + \delta_{gx} gx.
\end{array}
$
\end{enumerate*}
\end{theorem}

\begin{proof}
Note that only for the pairs $(1,1)$, $(1,g)$, $(g,1)$, $(g,g)$ is the Rota--Baxter identity nontrivial. Let $R(1) = \alpha_1 1 + \beta_1 g + \gamma_1 x + \delta_1 gx$, $R(g) = \alpha_g 1 + \beta_g g + \gamma_g x + \delta_g gx$. Thus, applying \eqref{idn:RB} to the pairs $(1,1)$ and $(g,g)$, we have
\begin{align*}
\left\{
\begin{array}{rll}
	\beta_1^2 - \alpha_1 (\alpha_1 + \lambda ) &=& 2 \beta_1 \alpha_g, \\
	\beta_1 (2\beta_g + \lambda) &=& 0, \\
	\lambda \gamma_1 + 2 \beta_1 \gamma_g &=& 0, \\
	\lambda \delta_1 + 2 \beta_1 \delta_g &=& 0.
\end{array}
\right.
\quad
\left\{
\begin{array}{rll}
	\alpha_g^2 + \alpha_1 (2\beta_g + \lambda) &=& \beta_g^2, \\
	\beta_1 (2\beta_g + \lambda) &=& 0, \\
	\gamma_1 (2\beta_g + \lambda) &=& 0, \\
	\delta_1 (2\beta_g + \lambda) &=& 0.
\end{array}
\right.
\end{align*}
Next, combining the Rota--Baxter identity for the pairs $(1,g)$ and $(g,1)$, we have
\begin{align*}
\left\{
\begin{array}{rll}
\beta_1 (\beta_g - \alpha_1) &=& (\alpha_1 + \beta_g + \lambda) \alpha_g, \\
\beta_1^2 + \beta_g (\beta_g + \lambda) &=& 0, \\
\beta_1 \gamma_1 + (\beta_g + \lambda) \gamma_g &=& 0, \\
\beta_1 \delta_1 + (\beta_g + \lambda) \delta_g &=& 0.
\end{array}
\right.
\end{align*}
Assume that $2 \beta_g + \lambda \ne 0$. Then it follows that $R(1) = -\lambda 1$. From the last system of equations we get $\beta_g (\beta_g + \lambda) = 0$. If $\beta_g = 0$, then $\gamma_g = \delta_g = 0$. Thus $R(1)$ and $R(g)$ are linearly dependent. This contradicts the fact that $\ker R = \langle x, gx \rangle$. So, $\beta_g = -\lambda$. Also, from the last system of equations we obtain $\alpha_g = 0$. Therefore, we get operator (1). 

Now suppose that $2 \beta_g + \lambda = 0$. Hence, we obtain $\alpha_g^2 = \beta_g^2$ and $\beta_1^2 = \lambda^2 / 4$. From the first system we obtain the equation for $\alpha_1$: $\alpha_1^2 + \lambda \alpha_1 - \lambda^2 / 4 + 2 \beta_1 \alpha_g = 0$. If $\beta_1 = \alpha_g$, then $\alpha_1 = -\lambda / 2$; hence, $R(1)$ and $R(g)$ are linearly dependent. Contradiction. 

Let $\beta_1 = -\alpha_g$. We obtain $\alpha_1 = \lambda / 2$ or $\alpha_1 = -\lambda - \lambda / 2$. Similarly, if $\alpha_1 = \lambda / 2$, then $R(1)$ and $R(g)$ are linearly dependent. So, $\alpha_1 = -\lambda - \lambda / 2$, and we get operators (2a) and (2b). Operators (2a) and (2b) lie in the same orbit under the action of the automorphism:
\begin{align*}
&f(g) = -g + \dfrac{4\gamma_{gx}}{\lambda}x, \
f(x) = x.
\end{align*}
In fact, we may assume \(\gamma_{gx} = \delta_{gx} = 0\) for operator (1), since under the automorphism
$$
\varphi(g) = g - \dfrac{\gamma_{gx}}{\lambda} x - \dfrac{\delta_{gx}}{\lambda} gx, \ \varphi(x) = x,
$$
the operator (1) with arbitrary \(\gamma_{gx}\) and \(\delta_{gx}\) is conjugate to the operator (1) with \(\gamma_{gx} = \delta_{gx} = 0\). The same reasoning applies to operator (2a); it suffices to consider the automorphism:
\begin{align*}
&\psi(g) = g - \dfrac{2\gamma_{gx}}{\lambda} x - \dfrac{2\delta_{gx}}{\lambda} gx, \ \psi(x) = x. \qedhere
\end{align*}
\end{proof}

\begin{theorem} \label{thr:1-g,x-gx}
Let $\ker R = \langle 1 - g, x - gx \rangle$. Then we have the following operators:\\*[\smallskipamount]
\begin{enumerate*}[label=\normalfont(\arabic*),itemjoin={}]
\item 
$
\begin{array}{lcl}
	R(g) &=&  -\frac{\lambda}{2} 1 - \frac{\lambda}{2} g + \gamma_g x - \gamma_g gx, \\
	R(gx) &=& -\frac{\lambda}{2} x - \frac{\lambda}{2} gx.
\end{array}
$ \quad

\item 
$
\begin{array}{lcl}
	R(g) &=&  \frac{\lambda}{2} 1 - \frac{\lambda}{2} g + \gamma_g x - \gamma_g gx, \\
	R(gx) &=& -\frac{\lambda}{2} x - \frac{\lambda}{2} gx.
\end{array}
$\\*[\medskipamount]

\item 
$
\begin{array}{lcl}
	R(g) &=& -\lambda 1, \\
	R(gx) &=& -\beta_{gx} 1 + \beta_{gx} g + \gamma_{gx} x - (\lambda + \gamma_{gx}) gx.\\
\end{array}
$
\end{enumerate*}
\end{theorem}

\begin{proof}
Let $R(g) = \alpha_g 1 + \beta_g g + \gamma_g x + \delta_g gx$ and $R(gx) = \alpha_{gx} 1 + \beta_{gx} g + \gamma_{gx} x + \delta_{gx} gx$. The Rota--Baxter identity, applied to the pair $(g,gx)$, gives us the following system:
\begin{align} \label{eq:g,gx} \tag{$*$}
\left\{
\begin{array}{rll}
(\alpha_g - \beta_g)\beta_{gx} + (\alpha_g + \beta_g + \lambda + \gamma_{gx} + \delta_{gx})\alpha_{gx} &=& 0, \\
\beta_{gx} (2\beta_g + \lambda + \gamma_{gx} + \delta_{gx}) &=& 0, \\
\beta_g \delta_{gx} - \beta_{gx} (\delta_g + \gamma_g) - \gamma_{gx} (\beta_g + \lambda + \gamma_{gx} + \delta_{gx}) &=& 0, \\
\beta_g \gamma_{gx} - \beta_{gx} (\delta_g + \gamma_g) - \delta_{gx} (\beta_g + \lambda + \gamma_{gx} + \delta_{gx}) &=& 0.
\end{array}
\right.
\end{align}
Applying \eqref{idn:RB} to the pairs $(g,g)$ and $(gx,gx)$ gives us 
\begin{align} \label{eq:g,g} \tag{$**$}
\left\{
\begin{array}{cll}
	\beta_g^2 - \alpha_g^2 &=& (2\beta_g + \lambda)\alpha_g, \\
	0 &=& (2\beta_g + \lambda)\beta_g, \\
	0 &=& (2\beta_g + \lambda)\gamma_g, \\
	0 &=& (2\beta_g + \lambda)\delta_g,
\end{array}
\right.
\quad \text{and} \quad
\alpha_{gx}^2 = \beta_{gx}^2. 
\end{align}
Assume that $2\beta_g + \lambda \ne 0$. Then, from the last system, $\beta_g = \gamma_g = \delta_g = 0$ and, consequently, $\alpha_g = -\lambda$. Let $\alpha_{gx} = \sigma\beta_{gx}$, $\sigma^2 = 1$. We simplify the first and second equations from \eqref{eq:g,gx}; we have $\beta_{gx}(\gamma_{gx} + \delta_{gx} + \lambda) = 0$ and $\beta_{gx}(\gamma_{gx} + \delta_{gx} - \sigma\lambda) = 0$. Hence, if $\sigma = 1$, then $\beta_{gx} = 0$. One of the elements \(\gamma_{gx}\) or \(\delta_{gx}\) is nonzero, since otherwise we would obtain a contradiction with the fact that \(\dim \ker R = 2\). Therefore, from \eqref{eq:g,gx}, we have $\gamma_{gx} + \delta_{gx} = -\lambda$. We obtained an operator which is a special case of operator (3). If $\sigma = -1$, then we obtain operator (3). 

Finally, let $2\beta_g + \lambda = 0$. Then, obviously, $\alpha_g^2 = \beta_g^2$; this follows from \eqref{eq:g,g}. The Rota--Baxter identity applied to the pair $(gx,g)$ gives us 
$$
(\beta_g - \alpha_g)\beta_{gx} + (\beta_g - \alpha_g + \lambda + \gamma_{gx} + \delta_{gx})\alpha_{gx} = 0.
$$
Subtracting the last equality from the first equality in \eqref{eq:g,gx} yields the following: $2\alpha_g(\alpha_{gx} + \beta_{gx}) + \lambda \beta_{gx} = 0$. So, let $\alpha_{gx} = \sigma\beta_{gx}$ and $\alpha_g = \mu\beta_g$, where $\sigma^2 = \mu^2 = 1$. Hence, we obtain $(\sigma + \mu + 1)\beta_{gx} = 0$. Therefore, $\beta_{gx} = 0$. Furthermore, from \eqref{eq:g,gx} we obtain:
$$
(\gamma_{gx} + \delta_{gx})(\gamma_{gx} + \delta_{gx} + \lambda) = 0, \quad
(\gamma_{gx} + \delta_{gx})(\gamma_{gx} - \delta_{gx}) = 0. 
$$
The Rota--Baxter identity applied to the pair $(x,1)$ gives us 
$$
\beta_g \delta_{gx} + \gamma_{gx}(\lambda - \beta_g + \gamma_{gx} + \delta_{gx}) = 0. 
$$
So, if $\gamma_{gx} + \delta_{gx} = 0$, then $-2\lambda \gamma_{gx} = 0$. Consequently, $\gamma_{gx} = 0$ and $\delta_{gx} = 0$. This contradicts the assumption that the dimension of the kernel of the operator is 2. Thus $\gamma_{gx} + \delta_{gx} \neq 0$ and $\gamma_{gx} = \delta_{gx} = -\lambda / 2$. We have almost obtained operators (1) and (2). It remains to observe that applying \eqref{idn:RB} to the pair \((1,1)\) implies that $(\gamma_g + \delta_g)R(gx) = 0$. So, $\gamma_g + \delta_g = 0$.

We note that the operators (1) and (2) with arbitrary \(\gamma_g\) are conjugate, under the automorphism
$$
\varphi(g) = g - \dfrac{2\gamma_g}{\lambda} x + \dfrac{2\gamma_g}{\lambda} gx, \ \varphi(x) = x, 
$$
to the operators (1) and (2) with \(\gamma_g = 0\).
\end{proof}

%
%

\section{Dimension of kernel is 1}

If the dimension of $\ker R$ is 1, then the dimension of $\mathrm{im} \, R$ is 3.  

\begin{theorem}
Let $\mathrm{im} \, R = \langle 1-g,x,gx \rangle$. Then we have the following operators:\\*[\smallskipamount]
\begin{enumerate*}[label=\normalfont(\arabic*),itemjoin={}]
\item
$
\begin{array}{lcllcl}
	R(1) &=& -\frac{\lambda }{2}(1-g) + \gamma_g x + \delta_g gx, & R(x) &=& -\lambda x, \\
	R(g) &=& \frac{\lambda }{2}(1-g) + \gamma_g x + \delta_g gx, & R(gx) &=& -\lambda gx.
\end{array}
$\\*[\medskipamount]

\item 
$
\begin{array}{lcllcl}
	R(1) &=& \frac{\lambda }{2}(1-g) + \gamma_g x + \delta_g gx, & R(x) &=& -\lambda x, \\
	R(g) &=& \frac{\lambda }{2}(1-g) - \gamma_g x - \delta_g gx, & R(gx) &=& -\lambda gx.
\end{array}
$\\*[\medskipamount]

\item 
$
\begin{array}{lcllcl}
	R(1) &=& 0, & R(x) &=& -\lambda x, \\
	R(g) &=& \lambda (1-g), & R(gx) &=& -\lambda gx.
\end{array}
$
\end{enumerate*} 
\end{theorem}

\begin{proof}
Let $R(1) = \alpha_1 (1 - g) + \gamma_1 x + \delta_1 gx$,  
$R(g) = \alpha_g (1 - g) + \gamma_g x + \delta_g gx$,  
$R(x) = \alpha_x (1 - g) + \gamma_x x + \delta_x gx$,  
$R(gx) = \alpha_{gx} (1 - g) + \gamma_{gx} x + \delta_{gx} gx$. Consider the Rota–Baxter identity \eqref{idn:RB} for the pair $(g, g)$:
\begin{align*}
R(g)R(g) &= R(R(g)g + gR(g) + \lambda 1), \\
2\alpha_g R(g) &= R(2\alpha_g g - 2\alpha_g 1 + \lambda 1), \\
0 &= (2\alpha_g - \lambda)R(1).
\end{align*}
We have two possibilities: either $2\alpha_g - \lambda = 0$, or $R(1) = 0$.  
Suppose first that $R(1) = 0$. Then $\ker R = k 1$.  
The Rota–Baxter identity for the pairs $(1, g)$, $(1, x)$, and $(1, gx)$ gives:
\begin{align*}
&R(R(g) + \lambda g) = 0, \quad R(R(x) + \lambda x) = 0, \quad R(R(gx) + \lambda gx) = 0, \\
\Longrightarrow&
\begin{cases*}
	R(g) + \lambda g = \alpha 1, \\ 
	R(x) + \lambda x = \beta 1, \\ 
	R(gx) + \lambda gx = \gamma 1,
\end{cases*}
\Longrightarrow
\begin{cases*}
	\alpha_g = \alpha = \lambda, \ \gamma_g = \delta_g = 0, \\
	\alpha_x = 0 = \beta, \ \gamma_x = -\lambda, \ \delta_x = 0, \\
	\alpha_{gx} = 0 = \gamma, \ \gamma_{gx} = 0, \ \delta_{gx} = -\lambda. 
\end{cases*} 
\end{align*} 
Thus, we obtain operator (3). Now suppose $2\alpha_g - \lambda = 0$. Considering the difference $R(g)R(x) - R(x)R(g)$, we get:
\begin{align*}
	\alpha_x \alpha_1 = 0, \quad \alpha_x \gamma_1 = 0, \quad \alpha_x \delta_1 = 0. 
\end{align*}
From this, it follows that $\alpha_x = 0$, since otherwise, if $\alpha_x \neq 0$, then $R(1) = 0$. Therefore, from $R(g) + \lambda g = \alpha 1$, we would obtain $-\lambda / 2 = -\lambda$, which is a contradiction.  
Hence, $\alpha_x = 0$. From the Rota–Baxter identity \eqref{idn:RB} for the pairs $(x, gx)$ and $(1, g)$, we obtain respectively:
$$
\alpha_{gx} = \alpha_x, \quad \text{and} \quad \alpha_1^2 = \lambda^2 / 4.
$$
Considering the identity for the pairs $(1, x)$ and $(x, 1)$, we obtain:
\begin{align*}
	\begin{cases*}
		R(R(x) + \lambda x) = \alpha_1(\gamma_{gx} + \delta_x) x 
		+ \alpha_1(\delta_{gx} - \gamma_x) gx, \\
		R(R(x) + \lambda x) = \alpha_1(\delta_x - \gamma_{gx}) x + 
		\alpha_1(\gamma_x - \delta_{gx}) gx, 
	\end{cases*} 
	\Longrightarrow
	\begin{cases*}
		\gamma_{gx} = 0, \\
		\gamma_x = \delta_{gx}.
	\end{cases*}
\end{align*}
Repeating the same reasoning with the pairs $(1, gx)$ and $(gx, 1)$, we find that $\delta_x = 0$.  
Thus we have:
$$
(\gamma_x + \lambda)R(x) = 0, \quad (\gamma_x + \lambda)R(gx) = 0.
$$
Since the kernel of the operator has dimension 1, it follows that $\gamma_x \ne 0$, and therefore $\gamma_x = -\lambda$. To complete the proof, consider the Rota–Baxter identity for the pair $(1, 1)$. We obtain:
$
2\alpha_1 \gamma_g + \lambda \gamma_1 = 0, \
2\alpha_1 \delta_g + \lambda \delta_1 = 0. 
$ Therefore, we get the operators (1) and (2).

Finally note that under the automorphism 
$$
\varphi(g) = g + \dfrac{2\gamma_g}{\lambda} x + \dfrac{2\delta_g}{\lambda} gx, \ \varphi(x) = x, 
$$
the operator (1) with arbitrary $\gamma_g$ and $\delta_g$ is conjugate to the operator (1) with $\gamma_g = \delta_g = 0$. The same holds for operator (2); it suffices to consider the automorphism
\begin{align*}
&\psi(g) = g - \dfrac{2\gamma_g}{\lambda} x - \dfrac{2\delta_g}{\lambda} gx, \ \psi(x) = x.\qedhere
\end{align*}
\end{proof}

\begin{theorem}
Let $\mathrm{im} \, R = \langle 1,x,gx \rangle$. Then we have the following operators:\\*[\smallskipamount]
\begin{enumerate*}[label=\normalfont(\roman*),itemjoin={}]
\item[\normalfont (1a)]
$
\begin{array}{lcllcl}
	R(1) &=& -\lambda 1, & R(x) &=& -\lambda x, \\
	R(g) &=& -\lambda 1 + \gamma_g x + \delta_g gx, & R(gx) &=& -\lambda gx.
\end{array}
$\\*[\medskipamount]

\item[\normalfont(1b)]
$
\begin{array}{lcllcl}
	R(1) &=& -\lambda 1, & R(x) &=& -\lambda x, \\
	R(g) &=& \lambda 1 + \gamma_g x + \delta_g gx, & R(gx) &=& -\lambda gx.
\end{array}
$
\end{enumerate*}
\end{theorem}

\begin{proof}
Let $R(1) = \alpha_1 1 + \gamma_1 x + \delta_1 gx$,  
$R(g) = \alpha_g 1 + \gamma_g x + \delta_g gx$,  
$R(x) = \alpha_x 1 + \gamma_x x + \delta_x gx$,  
$R(gx) = \alpha_{gx} 1 + \gamma_{gx} x + \delta_{gx} gx$. Consider the Rota–Baxter identity \eqref{idn:RB} for the pairs $(x, x)$ and $(gx, gx)$:  
$$
	R(x)R(x) = 2\alpha_x R(x), \quad R(gx)R(gx) = 2\alpha_{gx} R(gx).
$$  
It follows that $\alpha_x = \alpha_{gx} = 0$.  

Now consider the Rota–Baxter identity \eqref{idn:RB} for the pairs $(1, 1)$ and $(g, g)$. We obtain:  
\begin{align*}
	\begin{cases*}
		\alpha_g^2 + \lambda \alpha_1 = 0, \\
		\gamma_1 = \delta_1 = 0,
	\end{cases*} \quad 
	\alpha_1 (\alpha_1 + \lambda) = 0.
\end{align*}
Note that $\alpha_1 \ne 0$. Indeed, suppose $\alpha_1 = 0$; then $\ker R = k 1$ and $\alpha_g = 0$. In this case, either $R(g)$ is linearly dependent on $R(x)$ and $R(gx)$, or $R(x)$ and $R(gx)$ are linearly dependent. This contradicts the fact that the kernel has dimension 1. Therefore, $\alpha_1 = -\lambda$, so $\alpha_g^2 = \lambda^2$.

Finally, the Rota–Baxter identities for the pairs $(1, g)$, $(1, x)$, and $(1, gx)$ imply:  
\begin{align*}
	R(R(g) + \lambda g) = 0, \quad R(R(x) + \lambda x) = 0, \quad R(R(gx) + \lambda gx) = 0.
\end{align*}
Since the kernel has dimension 1, the elements $R(g) + \lambda g$ and $R(x) + \lambda x$ must be linearly dependent. Hence, $R(x) + \lambda x = 0$, since obviously there is no nonzero scalar that would make them linearly dependent otherwise. Similarly, $R(gx) + \lambda gx = 0$. Note that (1a) and (1b) lie in the same orbit under the action of the automorphism:
$$
f(g) = -g - \dfrac{2\gamma_g}{\lambda} x, \quad
f(x) = x. 
$$
Also note that the operator (1a) with arbitrary $\gamma_g$ and $\delta_g$ is conjugate to the operator (1a) with $\gamma_g = \delta_g = 0$ under the automorphism
$$
\varphi(g) = g + \dfrac{\gamma_g}{\lambda} x + \dfrac{\delta_g}{\lambda} gx, \ \varphi(x) = x. 
$$
This concludes the proof. 
\end{proof}

\begin{theorem}
Let $\mathrm{im} \, R = \langle 1,g,x-gx \rangle$. Then we have the following operators:\\*[\smallskipamount]
\begin{enumerate*}[label=\normalfont(\roman*),itemjoin={\\*[\medskipamount]}]
\item[\normalfont (1a)]
$
\begin{array}{lcllcl}
	R(1) &=& -\lambda 1, & R(x) &=& -\beta_{gx} 1 + \beta_{gx} g + \gamma_{gx} (x-gx), \\
	R(g) &=& -\lambda g, & R(gx) &=& -\beta_{gx} 1 + \beta_{gx} g + (\gamma_{gx}+\lambda) (x-gx).
\end{array}
$ 

\item[\normalfont (1b)] 
$
\begin{array}{lcllcl}
	R(1) &=& -\lambda 1, & R(x) &=& \beta_{gx} 1 + \beta_{gx} g + \gamma_{gx} (x-gx), \\
	R(g) &=& -\lambda g, & R(gx) &=& \beta_{gx} 1 + \beta_{gx} g + (\gamma_{gx}+\lambda) (x-gx).
\end{array}
$

\item[\normalfont (2a)]
$
\begin{array}{lcllcl}
	R(1) &=& -\frac{3\lambda}{2} 1 - \frac{\lambda}{2} g + \gamma_g (x-gx), & R(x) &=& -\frac{\lambda}{2}(x-gx), \\
	R(g) &=& \frac{\lambda}{2} 1 - \frac{\lambda}{2} g - \gamma_g (x-gx), & R(gx) &=& \frac{\lambda}{2}(x-gx). 
\end{array}
$

\item[\normalfont (2b)]
$
\begin{array}{lcllcl}
	R(1) &=& -\frac{3\lambda}{2} 1 + \frac{\lambda}{2} g + \gamma_g (x-gx), & R(x) &=& -\frac{\lambda}{2}(x-gx), \\
	R(g) &=& -\frac{\lambda}{2} 1 - \frac{\lambda}{2} g + \gamma_g (x-gx), & R(gx) &=& \frac{\lambda}{2}(x-gx).
\end{array}
$
\end{enumerate*}
\end{theorem}

\begin{proof}
Let  
$R(1) = \alpha_1 1 + \beta_1 g + \gamma_1 (x - gx)$,  
$R(g) = \alpha_g 1 + \beta_g g + \gamma_g (x - gx)$,  
$R(x) = \alpha_x 1 + \beta_x g + \gamma_x (x - gx)$,  
$R(gx) = \alpha_{gx} 1 + \beta_{gx} g + \gamma_{gx} (x - gx)$. Considering the Rota–Baxter identity \eqref{idn:RB} for the pairs $(g, g)$, $(x, x)$, and $(gx, gx)$, we obtain respectively:  
$$
(\beta_g^2 - \alpha_g^2) 1 = (2\beta_g + \lambda) R(1), \quad 
\alpha_x^2 = \beta_x^2, \quad 
\alpha_{gx}^2 = \beta_{gx}^2.
$$
We consider two cases: $2\beta_g + \lambda = 0$ and $2\beta_g + \lambda \neq 0$. Suppose first that $2\beta_g + \lambda \neq 0$. Then $R(1) = \alpha 1$, and from the Rota–Baxter identity for the pair $(1, 1)$, we conclude $\alpha(\alpha + \lambda) = 0$. If $\alpha = 0$, then $\ker R = k 1$, and $R(R(x) + \lambda x) = 0$, which implies $R(x) + \lambda x = \beta 1$, leading to $\gamma_x = -\lambda$, $\gamma_x = 0$, a contradiction. Therefore, $\alpha = -\lambda$, i.e., $R(1) = -\lambda 1$.  

Now, considering the Rota–Baxter identity for the pairs $(1, g)$ and $(1, x)$, we get:  
$$
R(R(g) + \lambda g) = 0, \quad R(R(x) + \lambda x) = 0.
$$  
Thus $\beta (R(g) + \lambda g) = \gamma (R(x) + \lambda x)$, where either $\beta \ne 0$ or $\gamma \ne 0$.  
If $\gamma \ne 0$, then $\beta \gamma_g = \gamma_x + \lambda$, $\beta \gamma_g = \gamma_x$, implying $\lambda = 0$, a contradiction.  
If $\beta \ne 0$, then $\gamma_g = \gamma \gamma_x + \gamma \lambda$, $\gamma_g = \gamma \gamma_x$, implying $\gamma = 0$, so $R(g) = -\lambda g$.  

Since $R(x) + \lambda x$ and $R(gx) + \lambda gx$ are linearly dependent, there exist $\widetilde{\alpha}, \widetilde{\beta}$ such that  
$$
\widetilde{\alpha}(R(x) + \lambda x) = \widetilde{\beta}(R(gx) + \lambda gx).
$$  
This leads to the system:  
\begin{align*}
	\begin{cases*}
		\widetilde{\alpha} \gamma_x + \widetilde{\alpha} \lambda = 
		\widetilde{\beta} \gamma_{gx}, \\
		-\widetilde{\alpha} \gamma_x = -\widetilde{\beta} \gamma_{gx} + \widetilde{\beta} \lambda,
	\end{cases*} \Longrightarrow
	\lambda (\widetilde{\alpha} - \widetilde{\beta}) = 0
	\Longrightarrow
	\widetilde{\alpha} = \widetilde{\beta}. 
\end{align*}  
Hence, $R(x) + \lambda x = R(gx) + \lambda gx$, which gives  
$
\gamma_{gx} = \gamma_x + \lambda, \ \alpha_x = \alpha_{gx}, \ \beta_x = \beta_{gx}.
$
We have thus obtained operators (1a) and (1b).  

Now suppose $2\beta_g + \lambda = 0$. Then $\alpha_g^2 = \beta_g^2$. Consider the Rota–Baxter identity for the pairs $(1, x)$ and $(g, x)$. We obtain:
\begin{multline*}
	\begin{cases*}
		\beta_x(\beta_1 - \alpha_g) = (\alpha_1 + \gamma_x + \lambda)\alpha_x + (\beta_1 - \gamma_x)\alpha_{gx}, \\
		\beta_x(\gamma_x + \frac{\lambda}{2}) + (\beta_1 - \gamma_x)\beta_{gx} = 0, \\
		\beta_x(\gamma_1 - \gamma_g) = \gamma_x(\beta_1 + \lambda + \gamma_x) + (\beta_1 - \gamma_x)\gamma_{gx},
	\end{cases*} \\
	\begin{cases*}
		\beta_x(\alpha_1 + \frac{\lambda}{2}) = (\gamma_x - \alpha_g)\alpha_x - (\gamma_x + \frac{\lambda}{2})\alpha_{gx}, \\
		\beta_x(\beta_1 - \gamma_x) + (\gamma_x + \frac{\lambda}{2})\beta_{gx} = 0, \\
		\beta_x(\gamma_g - \gamma_1) = (\gamma_x + \frac{\lambda}{2})(\gamma_{gx} - \gamma_x),
	\end{cases*}
\end{multline*}
respectively. Summing the second and third equations in each system yields:
\begin{align*}
	\begin{cases*}
		(2\beta_1 + \lambda)(\beta_x + \beta_{gx}) = 0, \\
		(2\beta_1 + \lambda)(\gamma_x + \gamma_{gx}) = 0.
	\end{cases*}
\end{align*}
Proceeding similarly with the pairs $(x, 1)$ and $(x, g)$, we obtain:
\begin{align*}
	\begin{cases*}
		(2\beta_1 - \lambda)(\beta_x + \beta_{gx}) = 0, \\
		(2\beta_1 - \lambda)(\gamma_x + \gamma_{gx}) = 0.
	\end{cases*}
\end{align*}
Therefore, $\beta_x + \beta_{gx} = 0$ and $\gamma_x + \gamma_{gx} = 0$. Consider the Rota–Baxter identity for the pair $(x, 1)$:
\begin{align*}
	\begin{cases*}
		\beta_x(\beta_1 - \alpha_g) = (\alpha_1 + \gamma_x + \lambda)\alpha_x - (\beta_1 + \gamma_x)\alpha_{gx}, \\
		\beta_{gx}(\gamma_x + \beta_1) - (\frac{\lambda}{2} + \gamma_x)\beta_x = 0, \\
		\beta_x(\gamma_1 + \gamma_g) = -\gamma_x(\lambda - \beta_1 + \gamma_x) + (\beta_1 + \gamma_x)\gamma_{gx}. 
	\end{cases*}
\end{align*}
Combining this with the earlier system for $(1, x)$, we deduce:  
$$
\beta_1 \alpha_{gx} = 0, \
\beta_1 \beta_{gx} = 0, \
\gamma_1 \beta_{gx} = 0.
$$  
Suppose that $\beta_1 = 0$. Then $\gamma_1 \neq 0$, since otherwise $R(1) = \alpha 1$ and it implies $R(g) = -\lambda g$, contradicting the fact that $2\beta_g + \lambda = 0$. So, $\beta_{gx} = \beta_x = \alpha_x = \alpha_{gx} = 0$. Considering the second coordinate of the \eqref{idn:RB} identity for the pair $(1, g)$, we have $\beta_g(\beta_g + \lambda) = 0$. Contradiction. Thus $\beta_1 \neq 0$. We have $\beta_{gx} = \beta_x = \alpha_x = \alpha_{gx} = 0$ and $\ker R = \langle x + gx \rangle$. Obviously, $\gamma_x \neq 0$, and it implies that $\gamma_x = -\lambda / 2$. Finally, the \eqref{idn:RB} identity for the pairs $(1,1)$ and $(1,g)$ yields the following:
$$
\begin{cases*}
(2\alpha_1 + \lambda)(\alpha_g + \beta_1) = 0, \\
\lambda^2 - 4\beta_1^2 = 0, \\
(\lambda - 2\beta_1)(\gamma_1 + \gamma_g) = 0, 
\end{cases*}
\quad 
\begin{cases*}
\alpha_1(\lambda + \alpha_1) + 2\alpha_g \beta_1 - \beta_1^2 = 0, \\
\lambda \gamma_1 - 2\beta_1 \gamma_g = 0. 
\end{cases*}
$$ 
Obviously, $2\alpha_1 + \lambda \neq 0$, since otherwise we would obtain a linear dependence between $R(1)$ and $R(g)$, which would imply that $\dim \ker R = 2$. Therefore, $\alpha_g = -\beta_1$ and $(\alpha_1 + \lambda / 2)^2 = \lambda^2$. We have obtained the operators (2a) and (2b).

The operators (1a) and (1b) are conjugated under antiautomorphism 
$
\psi(g) = -g, \ \psi(x) = -x.
$
The operators (2a) and (2b) are also conjugated under antiautomorphism 
$
\psi(g) = -g, \ \psi(x) = x.
$ 

Note that under automorphism 
$$
\varphi(g) = g - \dfrac{2\gamma_g}{\lambda} x + \dfrac{2\gamma_g}{\lambda} gx, \ \varphi(x) = x, 
$$
the operator (2a) with arbitrary $\gamma_g$ is conjugate to the operator (2a) with $\gamma_g = 0$.
\end{proof}

\section{Dimension of kernel is 0}

\begin{theorem}
Let $\ker R = 0$. Then we have the following operators:\\*[\smallskipamount]
\begin{enumerate*}[label=\normalfont(\arabic*),itemjoin={}]
\item 
$
\begin{array}{lcllcl}
R(1) &=& -\lambda 1, & R(x) &=& -\lambda x, \\
R(g) &=& -\lambda g, & R(gx) &=& -\lambda gx. 
\end{array}
$\\*[\medskipamount]

\item[\normalfont (2a)]
$
\begin{array}{lcllcl}
	R(1) &=& -\frac{3\lambda}{2} 1 - \frac{\lambda}{2} g + \gamma_g x + \delta_g gx, & R(x) &=& -\lambda x, \\
	R(g) &=& \frac{\lambda}{2} 1 - \frac{\lambda}{2} g - \gamma_g x - \delta_g gx, & R(gx) &=& -\lambda gx. 
\end{array}
$\\*[\medskipamount]

\item[\normalfont (2b)]
$
\begin{array}{lcllcl}
	R(1) &=& -\frac{3\lambda}{2} 1 + \frac{\lambda}{2} g + \gamma_g x + \delta_g gx, & R(x) &=& -\lambda x, \\
	R(g) &=& -\frac{\lambda}{2} 1 - \frac{\lambda}{2} g + \gamma_g x + \delta_g gx, & R(gx) &=& -\lambda gx. 
\end{array}
$
\end{enumerate*}
\end{theorem}

\begin{proof}
Since $(\beta_g^2 - \alpha_g^2) 1 = (2\beta_g + \lambda) R(1)$, by \eqref{idn:RB} applied to the pair $(g,g)$, we have two cases: either $2\beta_g + \lambda = 0$ or $2\beta_g + \lambda \neq 0$. If $2\beta_g + \lambda \neq 0$, then $R(1) = -\lambda 1$. Thus $R(g) + \lambda g = 0$, $R(x) + \lambda x = 0$, $R(gx) + \lambda gx = 0$. We have obtained operator (1).

Let $2\beta_g + \lambda = 0$. Then $\alpha_g^2 = \beta_g^2$. Consider the expressions $R(1)R(x) - R(x)R(1)$ and $R(1)R(gx) - R(gx)R(1)$. We have
$$
\begin{cases*}
\beta_1 \alpha_{gx} = 0, \\
\beta_1 \beta_{gx} = 0, \\
\beta_x \delta_1 = \beta_1 (\delta_x - \gamma_{gx}), \\
\beta_x \gamma_1 = \beta_1 (\gamma_x - \delta_{gx}), 
\end{cases*}
\quad
\begin{cases*}
\beta_1 \alpha_x = 0, \\
\beta_1 \beta_x = 0, \\
\beta_{gx} \delta_1 = \beta_1 (\delta_{gx} - \gamma_x), \\
\beta_{gx} \gamma_1 = \beta_1 (\gamma_{gx} - \delta_x). 
\end{cases*}
$$
Suppose that $\beta_1 = 0$. Since one of the elements $\gamma_1$ or $\delta_1$ is nonzero, we have $\beta_x = \beta_{gx} = 0$ and $\alpha_x = \alpha_{gx} = 0$. Considering the second coordinate of the \eqref{idn:RB} identity for the pair $(1, g)$, we have $\beta_g(\beta_g + \lambda) = 0$. Contradiction. Thus $\beta_1 \neq 0$. We have $\alpha_x = \beta_x = \alpha_{gx} = \beta_{gx} = 0$ and $\delta_x = \gamma_{gx}$, $\delta_{gx} = \gamma_x$. Comparing the expressions $R(1)R(x) - R(g)R(x)$ and $R(1)R(x) + R(g)R(x)$, we obtain
$$
(\gamma_x - \gamma_{gx})(\lambda + \gamma_x - \gamma_{gx}) = 0, \quad
(\gamma_x + \gamma_{gx})(\lambda + \gamma_x + \gamma_{gx}) = 0. 
$$
Clearly, $\gamma_x$ cannot equal $\gamma_{gx}$ or $-\gamma_{gx}$, since otherwise we would obtain a linear dependence between $R(x)$ and $R(gx)$. Thus, $\gamma_x = -\lambda$ and $\gamma_{gx} = 0$. Finally, the \eqref{idn:RB} identity for the pairs $(1,1)$ and $(1,g)$ yields the following:
\[
\begin{cases*}
(2\alpha_1 + \lambda)(\alpha_g + \beta_1) = 0, \\
\lambda^2 - 4\beta_1^2 = 0, \\
\lambda(\gamma_g - \delta_1) + \beta_1(\delta_g - \gamma_1) = 0, \\
\lambda \gamma_1 - 2\beta_1(\gamma_g - \delta_1) - \lambda \delta_g = 0,
\end{cases*}
\quad 
\begin{cases*}
\alpha_1(\lambda + \alpha_1) + 2\alpha_g \beta_1 - \beta_1^2 = 0, \\
\lambda \gamma_1 - 2\beta_1 \gamma_g = 0, \\
\lambda \delta_1 - 2\beta_1 \delta_g = 0. 
\end{cases*}
\]
Obviously, $2\alpha_1 + \lambda \neq 0$, since otherwise we would obtain a linear dependence between $R(1)$ and $R(g)$. Therefore, $\alpha_g = -\beta_1$ and $(\alpha_1 + \lambda / 2)^2 = \lambda^2$. We have obtained the operators (2a) and (2b), which are conjugate under the action of the automorphism:
\begin{align*}
&f(g) = -g + \dfrac{4\gamma_g}{\lambda}x, \quad
f(x) = -x. 
\end{align*}
Moreover note that the operator (2a) with arbitrary $\gamma_g$ and $\delta_g$ is conjugate to the operator (2a) with $\gamma_g = \delta_g = 0$ under the automorphism
\begin{align*}
&\varphi(g) = g - \dfrac{2\gamma_g}{\lambda} x - \dfrac{2\delta_g}{\lambda} gx, \ \varphi(x) = x. \qedhere
\end{align*}
\end{proof}

\section{Dualization}

Let us call the operator $-\lambda \operatorname{id} - R$ dual to the operator $R$. Most of the operators we have obtained are parameter‑free, as follows directly from the proofs. Consequently, computing their dual operators is straightforward and pleasant. However, there are also operators that depend on certain parameters. We now examine these parameter‑dependent operators and determine their duals. Consider the operator: 
$$
\begin{array}{lcllcl}
R(1) &=& 0, & R(x) &=& \alpha_x 1 + \alpha_x g - \lambda x, \\
R(g) &=& 0, & R(gx) &=& \alpha_x 1 + \alpha_x g - \lambda gx
\end{array}
$$
and its dual 
$$
\begin{array}{lcllcl}
R(1) &=& -\lambda 1, & R(x) &=& -\alpha_x 1 - \alpha_x g, \\
R(g) &=& -\lambda g, & R(gx) &=& -\alpha_x 1 - \alpha_x g.
\end{array}
$$
Obviously, if $\alpha_x = 0$, then the dual operator is one of the ones from Theorem \ref{thr:x,gx}. If $\alpha_x \neq 0$, then the dual operator has 2-dimensional kernel and its looks like $\langle 1+g+\gamma x + \delta gx, x-gx \rangle$. Its kernel is isomorphic to $\langle 1-g, x-gx \rangle$. So, the dual operator is conjugated with the operator (3) from Theorem \ref{thr:1-g,x-gx} with $\beta_{gx} \neq 0$. 

A similar situation occurs with the operators from Theorems \ref{thr:1,x} and \ref{thr:1,x-gx}: their duals are conjugate to the operator (3) in Theorem \ref{thr:1-g,x-gx} with \(\beta_{gx} = 0\) and either \(\gamma_{gx} \neq -\lambda/2\) or \(\gamma_{gx} = -\lambda/2\), respectively.

In conclusion, consider the operator (1) from Theorem \ref{thr:1-g,x-gx}: 
$$
\begin{array}{lcllcl}
R(1) &=& -\frac{\lambda}{2} 1 - \frac{\lambda}{2} g, & R(x) &=& -\frac{\lambda}{2} x - \frac{\lambda}{2} gx, \\
R(g) &=& -\frac{\lambda}{2} 1 - \frac{\lambda}{2} g, & R(gx) &=& -\frac{\lambda}{2} x - \frac{\lambda}{2} gx
\end{array}
$$
and its dual
$$
\begin{array}{lcllcl}
R(1) &=& -\frac{\lambda}{2} 1 + \frac{\lambda}{2} g, & R(x) &=& -\frac{\lambda}{2} x + \frac{\lambda}{2} gx, \\
R(g) &=& \frac{\lambda}{2} 1 - \frac{\lambda}{2} g, & R(gx) &=& \frac{\lambda}{2} x - \frac{\lambda}{2} gx.
\end{array}
$$
This operator dual to itself because the dual operator is conjugated with the operator (1) from Theorem \ref{thr:1-g,x-gx}. Finally, we present a complete list of non-trivial Rota--Baxter operators on the Sweedler algebra \(H_4\), up to conjugation and dualization.

\begin{theorem*}
Up to conjugation and dualization, the following non‑trivial Rota--Baxter operators on the Sweedler algebra are obtained:\\*[\medskipamount]
\begin{varwidth}{\textwidth}
\begin{enumerate}[label=\normalfont(\arabic*),itemjoin={\\*[\medskipamount]}]
\item
$
\begin{array}{lcllcl}
	R(1) &=& 0, & R(x) &=& \alpha_x 1 + \alpha_x g -\lambda x, \\
	R(g) &=& 0, & R(gx) &=& \alpha_x 1 + \alpha_x g -\lambda gx, \quad
	\alpha_x \neq 0.
\end{array}
$

\medskip

\item\quad
$
\begin{array}{lcllcl}
	R(1) &=& 0, & R(x) &=& 0, \\
	R(g) &=& -\lambda 1 - \lambda g, & R(gx) &=& -\lambda x - \lambda gx.
\end{array}
$

\medskip

\item
$
\begin{array}{lcllcl}
	R(1) &=& 0, & R(x) &=& -\frac{\lambda}{2} x - \frac{\lambda}{2} gx, \\
	R(g) &=& -\lambda 1 - \lambda g, & R(gx) &=& -\frac{\lambda}{2} x - \frac{\lambda}{2} gx.
\end{array}
$

\medskip

\item\quad
$
\begin{array}{lcllcl}
R(1) &=& -\lambda 1, & R(x) &=& 0, \\
R(g) &=& -\lambda g, & R(gx) &=& 0. 
\end{array}
$

\medskip

\item
$
\begin{array}{lcllcl}
	R(1) &=& -\frac{3\lambda}{2} 1 - \frac{\lambda}{2} g, & R(x) &=& 0,\\
	R(g) &=& \frac{\lambda}{2} 1 - \frac{\lambda}{2} g, & R(gx) &=& 0.
\end{array}
$
\end{enumerate}
\end{varwidth}\\*[\medskipamount]
\begin{varwidth}{\textwidth}
\begin{enumerate}[label=\normalfont(\arabic*),itemjoin={\\*[\medskipamount]}]
\setcounter{enumi}{5}
\item\quad
$
\begin{array}{lcllcl}
	R(1) &=& -\frac{\lambda}{2} 1 - \frac{\lambda}{2} g, & R(x) &=& -\frac{\lambda}{2} x - \frac{\lambda}{2} gx, \\
	R(g) &=& -\frac{\lambda}{2} 1 - \frac{\lambda}{2} g, & R(gx) &=& -\frac{\lambda}{2} x - \frac{\lambda}{2} gx.
\end{array}
$ 

\medskip

\item 
$
\begin{array}{lcllcl}
	R(1) &=& \frac{\lambda}{2} 1 - \frac{\lambda}{2} g, & R(x) &=& -\frac{\lambda}{2} x - \frac{\lambda}{2} gx, \\
	R(g) &=& \frac{\lambda}{2} 1 - \frac{\lambda}{2} g, & R(gx) &=& -\frac{\lambda}{2} x - \frac{\lambda}{2} gx.
\end{array}
$ 

\medskip

\item\quad
$
\begin{array}{lcllcl}
	R(1) &=& -\lambda 1, & R(x) &=& 0, \\
	R(g) &=& -\lambda 1, & R(gx) &=& 0.
\end{array}
$

\medskip

\item 
$
\begin{array}{lcllcl}
	R(1) &=& -\frac{\lambda}{2} 1 - \frac{\lambda}{2} g, & R(x) &=& 0, \\
	R(g) &=& -\frac{\lambda}{2} 1 - \frac{\lambda}{2} g, & R(gx) &=& 0.
\end{array}
$

\medskip

\item\quad
$
\begin{array}{lcllcl}
	R(1) &=& \frac{\lambda}{2} 1 - \frac{\lambda}{2} g, & R(x) &=& 0, \\
	R(g) &=& \frac{\lambda}{2} 1 - \frac{\lambda}{2} g, & R(gx) &=& 0.
\end{array}
$

\medskip

\item 
$
\begin{array}{lcllcl}
	R(1) &=& 0, & R(x) &=& 0, \\
	R(g) &=& -\lambda 1 - \lambda g, & R(gx) &=& 0.
\end{array}
$

\medskip 

\item\quad
$
\begin{array}{lcllcl}
	R(1) &=& -\lambda 1, & R(x) &=& \alpha_{gx} 1 - \alpha_{gx} g - \lambda x, \\
	R(g) &=& -\lambda 1, & R(gx) &=& \alpha_{gx} 1 - \alpha_{gx} g - \lambda gx, \quad \alpha_{gx} \neq 0.
\end{array}
$
\end{enumerate}
\end{varwidth}
\newpage
\begin{varwidth}{\textwidth}
\begin{enumerate}[label=\normalfont(\arabic*),itemjoin={\\*[\medskipamount]}]
\setcounter{enumi}{12}

\item 
$
\begin{array}{lcllcl}
	R(1) &=& -\lambda 1, & R(x) &=& -\frac{\lambda}{2} x - \frac{\lambda}{2} gx, \\
	R(g) &=& -\lambda g, & R(gx) &=& -\frac{\lambda}{2} x - \frac{\lambda}{2} gx.
\end{array}
$

\medskip

\item\quad
$
\begin{array}{lcllcl}
	R(1) &=& -\lambda 1, & R(x) &=& 0, \\
	R(g) &=& -\lambda g, & R(gx) &=& \lambda x - \lambda gx.
\end{array}
$
\end{enumerate}
\end{varwidth}
\end{theorem*}

\begin{corollary*}
The Rota--Baxter operators from {\normalfont \cite{ma:r-bo} (see also Introduction)} are such that:
\begin{enumerate}[label=\normalfont(\roman*),itemjoin={\\*[\medskipamount]}]
\item {\normalfont (a)} and {\normalfont (b)} are dual,

\medskip

\item {\normalfont (c)} is trivial,

\medskip

\item {\normalfont (d)} is conjugate to the operator {\normalfont (1)},

\medskip 

\item {\normalfont (e)} is conjugate to the operator {\normalfont (12)},

\medskip 

\item {\normalfont (f)} is conjugate to the dual of operator {\normalfont (3)},

\medskip

\item {\normalfont (g)} is conjugate to the dual of operator {\normalfont (1)},

\medskip

\item {\normalfont (h)} is conjugate to the operator {\normalfont (7)}.
\end{enumerate}
\end{corollary*}

\newpage

\begin{center}

\printbibliography

\end{center}
\end{document}